\documentclass[11pt]{amsart}

\newtheorem{theorem}{Theorem}[section]

\theoremstyle{definition}

\theoremstyle{remark}

\newtheorem{notation}[theorem]{Notation}
\numberwithin{equation}{section}

\begin{document}

\newcommand{\spacing}[1]{\renewcommand{\baselinestretch}{#1}\large\normalsize}
\spacing{1.14}

\title{Two new families of Finsler connections on even-dimensional manifolds}

\author {H. R. Salimi Moghaddam}

\address{Department of Mathematics\\ Faculty of  Sciences\\ University of Isfahan \\ Isfahan\\ 81746-73441-Iran.} \email{hr.salimi@sci.ui.ac.ir and salimi.moghaddam@gmail.com}



\begin{abstract}
Let $F^{2n}=(M,M',F^{\ast})$ be an even-dimensional pseudo-Finsler
manifold. We construct an almost hypercomplex structure on any
chart domain of a certain atlas of $M'$ by using a considered
non-linear connection. Then by using the almost hypercomplex
structure we define two new families of Finsler connections. Also
we show that for any Finsler connection $\nabla$ there exists a
linear connection $D$ such that the local almost hypercomplex
structure is parallel with respect to it.
\end{abstract}

\maketitle


\section{\textbf{Introduction}}

A. Bejancu and H.R. Farran, in \cite{BeFa} and \cite{BeFa1} , for
a pseudo-Finsler manifold $F^m=(M,M',F^{\ast})$ with a non-linear
connection $HM'$ and any two skew-symmetric Finsler tensor fields
of type $(1, 2)$ on $F^m$, introduced a notion of Finsler
connections which named "$(HM', S, T)-$Cartan connections". After
them in \cite{EsSa} we reconstruct the same Finsler connections by
using almost complex structures.\\
On the other hand almost hypercomplex and hypercomplex structures
which are important in differential geometry have many interesting
and effective applications in theoretical physics. For example the
background objects of HKT-geometry are hypercomplex manifolds.
These spaces appear in $N=4$ supersymmetric model (see
\cite{GiPaSt, Po}.). Applications of Riemannian metrics on these
spaces persuade us to study the geometry of Finsler metrics on
them (see \cite{Sa1, Sa2}).\\
In the present paper we study the relations between Finsler
structures and almost hypercomplex structures in a different
viewpoint. We use the almost hypercomplex structures to construct
new Finsler connections on even-dimensional pseudo-Finsler
manifolds . For this purpose we construct a local almost
hypercomplex structure by using a considered non-linear
connection. Then by using the almost hypercomplex structure we
define two new families of Finsler connections. Also we show that
for any Finsler connection $\nabla$ there exists a linear
connection $D$ such that the local almost hypercomplex structure
is parallel with respect to it.

\section{\textbf{Preliminaries and notations}}
Assume that $M$ is a real $m$-dimensional smooth manifold and $TM$
is the tangent bundle of $M$. Let $M'$ be a nonempty open
submanifold of $TM$ such that $\pi(M')=M$ and $\theta(M)\cap
M'=\emptyset$, where $\theta$ is the zero section of $TM$. Suppose
that $F^m=(M,M',F^{\ast})$ is a pseudo-Finsler manifold where
$F^{\ast}: M'\longrightarrow {\Bbb{R}}$ is a smooth function which
satisfies the following conditions in any coordinate system
$\{({\mathcal{U}}',\Phi'):x^i,y^i\}$ in $M'$, :

\begin{itemize}
    \item $F^\ast$ is positively homogeneous of degree two with
    respect to $(y^1,\dots,y^m)$, i.e., we have
    \begin{eqnarray*}
      F^\ast(x^1,\dots,x^m,ky^1,\dots,ky^m)=k^2F^\ast(x^1,\dots,x^m,y^1,\dots,y^m)
    \end{eqnarray*}
    for any point $(x,y)\in({\mathcal{U}}',\Phi')$ and $k>0$.
    \item At any point $(x,y)\in({\mathcal{U}}',\Phi')$, $g_{ij}$
    are the components of a quadratic form on ${\Bbb{R}}^m$ with
    $q$ negative eigenvalues and $m-q$ positive eigenvalues,
    $0<q<m$ (see \cite{BeFa}).
\end{itemize}

Consider the tangent mapping $\pi_\ast: TM'\longrightarrow TM$ of
the projection map $\pi:M'\longrightarrow M$ and def\/ine the
vector bundle $VM'=\ker \pi_\ast$. A complementary distribution
$HM'$ to $VM'$ in $TM'$ is called a non-linear connection or a
horizontal distribution on $M'$
\begin{eqnarray}
  TM'=HM'\oplus VM'.
\end{eqnarray}

In local coordinates let $\{\frac{\delta}{\delta
x^i}=\frac{\partial}{\partial
x^i}-N^j_i(x,y)\frac{\partial}{\partial y^j},
\frac{\partial}{\partial y^i} \}, \ \ (i,j=1\cdots m)$ be the
canonical basis for $TM'$. Let $F^m=(M,M',F^{\ast})$ be a
pseudo-Finsler manifold. Then a Finsler connection on $F^m$ is a
pair $FC=(HM',\nabla)$ where $HM'$ is a non-linear connection on
$M'$ and $\nabla$ is a linear connection
on the vertical vector bundle $VM'$ (see \cite{BeFa}).\\

An almost hypercomplex manifold is a $4n$-dimensional manifold $M$
with three globally-defined almost complex structures $J_i,
i=1,2,3,$ satisfying the quaternion identities
\begin{eqnarray}
  J_1J_2&=& -J_2J_1=J_3, \label{JJ} \\
  J_i^2 &=& -Id_{TM},  \ \ \ \ \ \ \ i=1,2,3.
\end{eqnarray}

\begin{notation}
From now on we suppose that $F^{2n}=(M,M',F^{\ast})$ is an
even-dimensional pseudo-Finsler manifold. We use $h$ and $v$ for
the projections on $HM'$ and $VM'$, respectively. Also we use
$\delta_i$ and $\partial_i$ as $\frac{\delta}{\delta x_i}$ and
$\frac{\partial}{\partial y_i}$, respectively.

Throughout the article, we use the following rules for the
indices.
\begin{itemize}
    \item The indicant $a$ is equal to $1$ and $3$, only.
    \item Latin indices (except the alphabets $a$ and $n$) run from $1$ to
    $2n$.
    \item Greek indices $\alpha, \beta, \gamma$ and $\theta$ run from $1$ to $n$.
    \item Any repeated pair of indices (except the alphabets $a$ and $n$), provided that one is up
    and the other is down, is automatically summed.
    \item The matrix $(g^{ij})$ is the matrix inverse of
    $(g_{ij})$.
\end{itemize}

\end{notation}

\section{\textbf{New Finsler connections on even-dimensional manifolds}}

In this section we construct two new families of Finsler
connections on $M'$. For this reason we consider an atlas on $M'$,
also we suppose that $HM'$ is a non-linear connection on $M'$.
Then by using the non-linear connection we define an almost
hypercomplex structure on any chart. After this step we define new
Finsler connections on any chart by using the almost hapercomplex
structure . Finally it is enough to paste the local connections by
a partition of unity to have a connection on $M'$.\\

Let $\mathcal{A}$ be an atlas on $M'$ and
$({\mathcal{U}}',\Phi')\in\mathcal{A}$. We construct the following
almost hypercomplex structure on ${\mathcal{U}}'$ by using
non-linear connection $HM'$ .
\begin{eqnarray}
 \begin{array}{ccc}\label{J}
  J_1(\delta_\alpha)=\partial_\alpha & J_2(\delta_\alpha)=\delta_{n+\alpha} & J_3(\delta_\alpha)=-\partial_{n+\alpha} \\
  J_1(\delta_{n+\alpha})=-\partial_{n+\alpha} & J_2(\delta_{n+\alpha})=-\delta_\alpha & J_3(\delta_{n+\alpha})=-\partial_\alpha \\
  J_1(\partial_\alpha)=-\delta_\alpha & J_2(\partial_\alpha)=\partial_{n+\alpha} & J_3(\partial_\alpha)=\delta_{n+\alpha} \\
  J_1(\partial_{n+\alpha})= \delta_{n+\alpha}& J_2(\partial_{n+\alpha})=-\partial_\alpha & J_3(\partial_{n+\alpha})=\delta_\alpha \\
\end{array}
\end{eqnarray}
It is easy to show
\begin{eqnarray}
  J_1^2&=&J_2^2 = J_3^2 = -Id_{T{\mathcal{U}}'} \\
  J_3 &=& J_1\circ J_2 = -J_2\circ J_1
\end{eqnarray}
Now by using $J_1$ and $J_3$ we construct two linear connections
on ${\mathcal{U}}'$ as follows.
\begin{theorem}\label{D}
Let $\nabla$ be a Finsler connection on ${\mathcal{U}}'$. The
differential operators $D^a$, $a=1, 3$, which are defined by
\begin{equation}\label{D13}
    D_X^aY:=\nabla_XvY-J_a\nabla_XJ_ahY \ \ \ \ \ \ \ \ \ \forall
    X,Y\in\Gamma(T{\mathcal{U}}')
\end{equation}
are two linear connections on ${\mathcal{U}}'$. Also $J_1$ and
$J_3$ are parallel with respect to $D^1$ and $D^3$, respectively.
\end{theorem}
\begin{proof}
The proof is easy so we omit it.
\end{proof}
Let $T^{D^a}$, $a=1, 3$, denotes the torsion tensor of $D^a$. It
is simple to see,
\begin{eqnarray}\label{torsion of D13}
  T^{D^a}(X,Y)&=&(\nabla_XvY-\nabla_YvX-v[X,Y])-\nonumber\\
  &&J_a(\nabla_XJ_ahY-\nabla_YJ_ahX-J_ah[X,Y]) \ \ \  \forall
    X,Y\in\Gamma(T{\mathcal{U}}').
\end{eqnarray}

Now we can introduce two families of new Finsler connections by using $J_1$ and $J_3$.
\begin{theorem}
Suppose that $HM'$ is a non-linear connection on $M'$ and $S$ and
$T$ are two arbitrary skew-symmetric Finsler tensor fields of type
$(1,2)$ on $F^{2n}$. Then there exists a unique linear connection
$\nabla^a$ on $V{\mathcal{U}}'$ satisfying the conditions:
\begin{enumerate}
    \item $\nabla^a$ is a metric connection.
    \item $T^{D^a}$, $S$ and $T$ satisfy,\begin{description}
        \item[i] $T^{D^a}(vX,vY)=S(vX,vY)$,
        \item[ii] $hT^{D^a}(hX,hY)=J_aT(J_ahY,J_ahX)$.
    \end{description}
\end{enumerate}
\end{theorem}
\begin{proof}
For any $X,Y,Z\in\Gamma(T{\mathcal{U}}')$ we define a linear
connection $\nabla^a$ on $V{\mathcal{U}}'$ by the following
equations.

\begin{eqnarray}
  2g(\nabla^a_{vX}vY,vZ)&=& vX(g(vY,vZ))+vY(g(vZ,vX))-vZ(g(vX,vY))\nonumber\\
  &&+ g(vY,[vZ,vX])+g(vZ,[vX,vY])-g(vX,[vY,vZ])+g(vY,S(vZ,vX))\nonumber\\
  &&+g(vZ,S(vX,vY))-g(vX,S(vY,vZ))\label{eq1}
\end{eqnarray}
and
\begin{eqnarray}
   2g(\nabla^a_{hX}J_ahY, J_ahZ)&=& hX(g(J_ahY,J_ahZ))+hY(g(J_ahZ,J_ahX))\nonumber\\
                            &&- hZ(g(J_ahX,J_ahY))+g(J_ahY,J_ah[hZ,hX])+ g(J_ahZ,J_ah[hX,hY]) \label{eq2}\\
                            &&-g(J_ahX,J_ah[hY,hZ])+g(J_ahY,T(J_ahZ,J_ahX))\nonumber  \\
                            &&+g(J_ahZ,T(J_ahX,J_ahY))-g(J_ahX,T(J_ahY,J_ahZ)).\nonumber
\end{eqnarray}
Now we show that $g$ is parallel with respect to $\nabla^a$. It is
easy to see $J_a\circ v=h\circ J_a$ and $v\circ J_a=J_a\circ h$.
After performing some computations for any
$X,Y,Z\in\Gamma(T{\mathcal{U}}')$ we have,
\begin{eqnarray}
  (\nabla^a_Xg)(vY,vZ) &=& (\nabla^a_{vX+hX}g)(vY,vZ)\nonumber\\
                     &=& vX(g(vY,vZ))-g(\nabla^a_{vX}vY,vZ)-g(vY,\nabla^a_{vX}vZ)+hX(g(vY,vZ))\nonumber\\
                     &&-g(\nabla^a_{hX}vY,vZ)-g(vY,\nabla^a_{hX}vZ)=0.\nonumber
\end{eqnarray}

So $\nabla^a$ is a metric connection.\\
Locally we set
\begin{eqnarray}
  \nabla^1_{\delta_j}\partial_i=F^k_{ij}\partial_k && \nabla^3_{\delta_j}\partial_i=\tilde{F}^k_{ij}\partial_k\\
  \nabla^1_{\partial_j}\partial_i=C^k_{ij}\partial_k && \nabla^3_{\partial_j}\partial_i=\tilde{C}^k_{ij}\partial_k\\
  S(\partial_j,\partial_i)=S^k_{ij}\partial_k&&T(\partial_j,\partial_i)=T^k_{ij}\partial_k.
\end{eqnarray}
In the relation \ref{eq1} let $X=\partial_j$, $Y=\partial_i$ and $Z=\partial_l$, then we can obtain the coefficients
$C^k_{ij}$ and $\tilde{C}^k_{ij}$ as follows:
\begin{eqnarray}
  \label{eq3} \hspace*{1.5cm} \tilde{C}^k_{ij}=C^k_{ij}=\frac{1}{2}g^{lk}\{\partial_j g_{il}+\partial_i g_{lj}-\partial_l g_{ji}
  +S^h_{jl}g_{ih}+S^h_{ij}g_{lh}-S^h_{li}g_{jh}\}.
\end{eqnarray}
By attention to the relation \ref{J}, for computing the
coefficients $F^k_{ij}$ and $\tilde{F}^k_{ij}$ we must consider
four cases for any connection $\nabla^a$ as follows:
\begin{eqnarray}
  && F^k_{\alpha \beta}, F^k_{\alpha \ n+\beta}, F^k_{n+\alpha \ \beta}, F^k_{n+\alpha \ n+\beta} \\
  && \tilde{F}^k_{\alpha \beta}, \tilde{F}^k_{\alpha \ n+\beta}, \tilde{F}^k_{n+\alpha \ \beta}, \tilde{F}^k_{n+\alpha \
  n+\beta}.
\end{eqnarray}
\textbf{The Computation of $F^k_{\alpha \beta}$ and $\tilde{F}^k_{n+\alpha
\beta}$.}\\
In the equation \ref{eq2} let $X=\delta_\beta$, $Y=\delta_\alpha$ and $Z=\delta_{\theta}$.
Then we have,
\begin{eqnarray}
  2F^k_{\alpha\beta} g_{k\theta} &=& \delta_{\beta} g_{\alpha\theta}+\delta_{\alpha} g_{\theta\beta}-\delta_{\theta} g_{\beta\alpha}
             +T^h_{\beta\theta} g_{\alpha h}+T^h_{\alpha\beta} g_{\theta h}-T^h_{\theta\alpha} g_{\beta h} \label{XYZ111}\\
  2\tilde{F}^k_{n+\alpha \ \beta}  g_{k \ n+\theta} &=& \delta_{\beta} g_{n+\alpha \ n+\theta}+\delta_{\alpha} g_{n+\theta \ n+\beta}-\delta_{\theta} g_{n+\beta \
             n+\alpha}\label{XYZ113}\\
  && -T^h_{n+\beta \ n+\theta} g_{n+\alpha \ h}-T^h_{n+\alpha \ n+\beta} g_{n+\theta \ h}+T^h_{n+\theta \ n+\alpha} g_{n+\beta \
             h}\nonumber.
\end{eqnarray}
Now for the same $X$ and $Y$ let $Z=\delta_{n+\theta}$, after some computations we have
\begin{eqnarray}
  2F^k_{\alpha\beta} g_{k \ n+\theta} &=& \delta_{\beta} g_{\alpha \ n+\theta}+\delta_{\alpha} g_{n+\theta \ \beta}+\delta_{n+\theta}
             g_{\beta\alpha}\label{XYZ121}\\
  &&+T^h_{\beta \ n+\theta} g_{\alpha h}+T^h_{\alpha\beta} g_{n+\theta h}-T^h_{n+\theta \ \alpha} g_{\beta h} \nonumber\\
  2\tilde{F}^k_{n+\alpha \ \beta}  g_{k \theta} &=& \delta_{\beta} g_{n+\alpha \ \theta}+\delta_{\alpha} g_{\theta \ n+\beta}-\delta_{n+\theta} g_{n+\beta \
             n+\alpha}\label{XYZ123}\\
  && -T^h_{n+\beta \ \theta} g_{n+\alpha \ h}-T^h_{n+\alpha \ n+\beta} g_{\theta h}+T^h_{\theta \ n+\alpha} g_{n+\beta \
             h}\nonumber.
\end{eqnarray}
Now the relations (\ref{XYZ111},\ref{XYZ121}) and the relations
(\ref{XYZ113},\ref{XYZ123}) respectively show that
\begin{eqnarray}
  F^k_{\alpha\beta} &=& \frac{1}{2}g^{kl}\{\delta_\beta g_{\alpha l}+\delta_\alpha g_{l\beta}
                         +T^h_{\beta l}g_{\alpha h}+T^h_{\alpha\beta}g_{lh}-T^h_{l\alpha}g_{\beta h}\} \\
                         &&-\frac{1}{2}(g^{k \gamma}\delta_{\gamma}-g^{k \ n+\gamma}\delta_{n+\gamma})g_{\beta\alpha}\nonumber\\
  \tilde{F}^k_{n+\alpha \ \beta} &=& \frac{1}{2}g^{kl}\{\delta_\beta g_{n+\alpha \ l}+\delta_\alpha g_{l \ n+\beta}
                         -T^h_{n+\beta \ l}g_{n+\alpha \ h}-T^h_{n+\alpha \ n+\beta}g_{lh}\\
                         &&+T^h_{l \ n+\alpha}g_{n+\beta \ h}\}-\frac{1}{2}(g^{k \gamma}\delta_{n+\gamma}+g^{k \ n+\gamma}\delta_{\gamma})g_{n+\beta \ n+\alpha}\nonumber.
\end{eqnarray}
\textbf{The Computation of $F^k_{\alpha \ n+\beta}$ and $\tilde{F}^k_{n+\alpha
 \ n+\beta}$.}\\
Similar to the pervious case in the equation \ref{eq2}, let
$X=\delta_{n+\beta}$, $Y=\delta_\alpha$ and $Z=\delta_{\theta}$.
Then we have,
\begin{eqnarray}
  2F^k_{\alpha \ n+\beta} g_{k\theta} &=& \delta_{n+\beta} g_{\alpha\theta}-\delta_{\alpha} g_{\theta \ n+\beta}+\delta_{\theta} g_{n+\beta \ \alpha}\label{XYZ211}\\
             &&-T^h_{n+\beta \ \theta} g_{\alpha h}-T^h_{\alpha \ n+\beta} g_{\theta h}+T^h_{\theta\alpha} g_{n+\beta \ h}\nonumber \\
  2\tilde{F}^k_{n+\alpha \ n+\beta}  g_{k \ n+\theta} &=& \delta_{n+\beta} g_{n+\alpha \ n+\theta}+\delta_{\alpha} g_{n+\theta \ \beta}-\delta_{\theta} g_{\beta \ n+\alpha}\label{XYZ213}\\
  && -T^h_{\beta \ n+\theta} g_{n+\alpha \ h}-T^h_{n+\alpha \ \beta} g_{n+\theta \ h}+T^h_{n+\theta \ n+\alpha} g_{\beta \ h}\nonumber.
\end{eqnarray}
Now for $X=\delta_{n+\beta}$, $Y=\delta_\alpha$ let
$Z=\delta_{n+\theta}$, then we have
\begin{eqnarray}
  2F^k_{\alpha \ n+\beta} g_{k \ n+\theta} &=& \delta_{n+\beta} g_{\alpha \ n+\theta}-\delta_{\alpha} g_{n+\theta \ n+\beta}-\delta_{n+\theta} g_{n+\beta \ \alpha}\label{XYZ221}\\
  &&-T^h_{n+\beta \ n+\theta} g_{\alpha h}-T^h_{\alpha \ n+\beta} g_{n+\theta h}+T^h_{n+\theta \ \alpha} g_{n+\beta h} \nonumber\\
  2\tilde{F}^k_{n+\alpha \ n+\beta}  g_{k \theta} &=& \delta_{n+\beta} g_{n+\alpha \ \theta}+\delta_{\alpha} g_{\theta \ \beta}-\delta_{n+\theta} g_{\beta \ n+\alpha}\label{XYZ223}\\
  && -T^h_{\beta \ \theta} g_{n+\alpha \ h}-T^h_{n+\alpha \ \beta} g_{\theta h}+T^h_{\theta \ n+\alpha} g_{\beta \ h}\nonumber.
\end{eqnarray}
Now the relations (\ref{XYZ211},\ref{XYZ221}) and the relations
(\ref{XYZ213},\ref{XYZ223}) respectively show that
\begin{eqnarray}
  F^k_{\alpha \ n+\beta} &=& \frac{1}{2}g^{kl}\{\delta_{n+\beta} g_{\alpha l}-\delta_\alpha g_{l \ n+\beta}
                         -T^h_{n+\beta \ l}g_{\alpha h}-T^h_{\alpha \ n+\beta}g_{lh}+T^h_{l\alpha}g_{n+\beta h}\} \\
                         &&+\frac{1}{2}(g^{k \gamma}\delta_{\gamma}-g^{k \ n+\gamma}\delta_{n+\gamma})g_{n+\beta \ \alpha}\nonumber\\
  \tilde{F}^k_{n+\alpha \ n+\beta} &=& \frac{1}{2}g^{kl}\{\delta_{n+\beta} g_{n+\alpha \ l}+\delta_\alpha g_{l \ \beta}
                         -T^h_{\beta \ l}g_{n+\alpha \ h}-T^h_{n+\alpha \ \beta}g_{lh}\\
                         &&+T^h_{l \ n+\alpha}g_{\beta \ h}\}-\frac{1}{2}(g^{k \gamma}\delta_{n+\gamma}+g^{k \ n+\gamma}\delta_{\gamma})g_{\beta \ n+\alpha}\nonumber.
\end{eqnarray}
\textbf{The Computation of $F^k_{n+\alpha \ \beta}$ and $\tilde{F}^k_{\alpha\beta}$.}\\
In \ref{eq2} consider $X=\delta_{\beta}$, $Y=\delta_{n+\alpha}$
and $Z=\delta_{\theta}$. Then we have
\begin{eqnarray}
  2F^k_{n+\alpha \ \beta} g_{k\theta} &=& \delta_{\beta} g_{n+\alpha \ \theta}-\delta_{n+\alpha} g_{\theta \ \beta}-\delta_{\theta} g_{\beta \ n+\alpha}\label{XYZ311}\\
             &&+T^h_{\beta \ \theta} g_{n+\alpha h}+T^h_{n+\alpha \ \beta} g_{\theta h}-T^h_{\theta \ n+\alpha} g_{\beta \ h}\nonumber \\
  2\tilde{F}^k_{\alpha\beta}  g_{k \ n+\theta} &=& \delta_{\beta} g_{\alpha \ n+\theta}+\delta_{n+\alpha} g_{n+\theta \ n+\beta}-\delta_{\theta} g_{n+\beta \ \alpha}\label{XYZ313}\\
  && -T^h_{n+\beta \ n+\theta} g_{\alpha \ h}-T^h_{\alpha \ n+\beta} g_{n+\theta \ h}+T^h_{n+\theta \ \alpha} g_{n+\beta \ h}\nonumber.
\end{eqnarray}
Now for the same $X$ and $Y$ let $Z=\delta_{n+\theta}$, then we
have
\begin{eqnarray}
  2F^k_{n+\alpha \ \beta} g_{k \ n+\theta} &=& \delta_{\beta} g_{n+\alpha \ n+\theta}-\delta_{n+\alpha} g_{n+\theta \ \beta}+\delta_{n+\theta} g_{\beta \ n+\alpha}\label{XYZ321}\\
  &&+T^h_{\beta \ n+\theta} g_{n+\alpha h}+T^h_{n+\alpha \ \beta} g_{n+\theta h}-T^h_{n+\theta \ n+\alpha} g_{\beta h} \nonumber\\
  2\tilde{F}^k_{\alpha \ \beta}  g_{k \theta} &=& \delta_{\beta} g_{\alpha \ \theta}+\delta_{n+\alpha} g_{\theta \ n+\beta}-\delta_{n+\theta} g_{n+\beta \ \alpha}\label{XYZ323}\\
  && -T^h_{n+\beta \ \theta} g_{\alpha \ h}-T^h_{\alpha \ n+\beta} g_{\theta h}+T^h_{\theta \ \alpha} g_{n+\beta \ h}\nonumber.
\end{eqnarray}
Equations (\ref{XYZ311},\ref{XYZ321}) and equations
(\ref{XYZ313},\ref{XYZ323}) respectively show that
\begin{eqnarray}
  F^k_{n+\alpha \ \beta} &=& \frac{1}{2}g^{kl}\{\delta_{\beta} g_{n+\alpha \ l}-\delta_{n+\alpha} g_{l\beta}
                         +T^h_{\beta l}g_{n+\alpha \ h}+T^h_{n+\alpha \ \beta}g_{lh}-T^h_{l \ n+\alpha}g_{\beta h}\} \\
                         &&+\frac{1}{2}(g^{k \ n+\gamma}\delta_{n+\gamma}-g^{k \ \gamma}\delta_{\gamma})g_{\beta \ n+\alpha}\nonumber\\
  \tilde{F}^k_{\alpha\beta} &=& \frac{1}{2}g^{kl}\{\delta_{\beta} g_{\alpha l}+\delta_{n+\alpha} g_{l \ n+\beta}
                         -T^h_{n+\beta \ l}g_{\alpha h}-T^h_{\alpha \ n+\beta}g_{lh}\\
                         &&+T^h_{l\alpha}g_{n+\beta \ h}\}-\frac{1}{2}(g^{k \gamma}\delta_{n+\gamma}+g^{k \ n+\gamma}\delta_{\gamma})g_{n+\beta \ \alpha}\nonumber.
\end{eqnarray}
\textbf{The Computation of $F^k_{n+\alpha \ n+\beta}$ and $\tilde{F}^k_{\alpha\ n+\beta}$.}\\
Now for the last time in equation \ref{eq2}, let
$X=\delta_{n+\beta}$, $Y=\delta_{n+\alpha}$ and
$Z=\delta_{\theta}$. Then we have,
\begin{eqnarray}
  2F^k_{n+\alpha \ n+\beta} g_{k\theta} &=& \delta_{n+\beta} g_{n+\alpha \ \theta}+\delta_{n+\alpha} g_{\theta \ n+\beta}+\delta_{\theta} g_{n+\beta \ n+\alpha}\label{XYZ411}\\
             &&-T^h_{n+\beta \ \theta} g_{n+\alpha h}-T^h_{n+\alpha \ n+\beta} g_{\theta h}-T^h_{\theta \ n+\alpha} g_{n+\beta \ h}\nonumber \\
  2\tilde{F}^k_{\alpha \ n+\beta}  g_{k \ n+\theta} &=& \delta_{n+\beta} g_{\alpha \ n+\theta}+\delta_{n+\alpha} g_{n+\theta \ \beta}-\delta_{\theta} g_{\beta\alpha}\label{XYZ413}\\
  && -T^h_{\beta \ n+\theta} g_{\alpha h}-T^h_{\alpha\beta} g_{n+\theta \ h}+T^h_{n+\theta \ \alpha} g_{\beta h}\nonumber.
\end{eqnarray}
Similar to the pervious cases for the same $X$ and $Y$, let
$Z=\delta_{n+\theta}$, after some computations we have
\begin{eqnarray}
  2F^k_{n+\alpha \ n+\beta} g_{k \ n+\theta} &=& \delta_{n+\beta} g_{n+\alpha \ n+\theta}+\delta_{n+\alpha} g_{n+\theta \ n+\beta}-\delta_{n+\theta} g_{n+\beta \ n+\alpha}\label{XYZ421}\\
  &&-T^h_{n+\beta \ n+\theta} g_{n+\alpha h}-T^h_{n+\alpha \ n+\beta} g_{n+\theta h}+T^h_{n+\theta \ n+\alpha} g_{n+\beta h} \nonumber\\
  2\tilde{F}^k_{\alpha \ n+\beta}  g_{k\theta} &=& \delta_{n+\beta} g_{\alpha\theta}+\delta_{n+\alpha} g_{\theta\beta}-\delta_{n+\theta} g_{\beta\alpha}\label{XYZ423}\\
  && -T^h_{\beta\theta} g_{\alpha h}-T^h_{\alpha\beta} g_{\theta h}+T^h_{\theta\alpha} g_{\beta h}\nonumber.
\end{eqnarray}
Now by the equations (\ref{XYZ411},\ref{XYZ421}) and
(\ref{XYZ413},\ref{XYZ423}) we respectively have
\begin{eqnarray}
  F^k_{n+\alpha \ n+\beta} &=& \frac{1}{2}g^{kl}\{\delta_{n+\beta} g_{n+\alpha \ l}+\delta_{n+\alpha} g_{l \ n+\beta}
                         -T^h_{n+\beta \ l}g_{n+\alpha \ h}\nonumber\\
                         &&-T^h_{n+\alpha \ n+\beta}g_{lh}+T^h_{l \ n+\alpha}g_{n+\beta \ h}\} \\
                         &&+\frac{1}{2}(g^{k\gamma}\delta_{\gamma}-g^{k \ n+\gamma}\delta_{n+\gamma})g_{n+\beta \ n+\alpha}\nonumber\\
  \tilde{F}^k_{\alpha \ n+\beta} &=& \frac{1}{2}g^{kl}\{\delta_{n+\beta} g_{\alpha l}+\delta_{n+\alpha} g_{l\beta}
                         -T^h_{\beta l}g_{\alpha h}-T^h_{\alpha\beta}g_{lh}\\
                         &&+T^h_{l\alpha}g_{\beta h}\}-\frac{1}{2}(g^{k \gamma}\delta_{n+\gamma}+g^{k \ n+\gamma}\delta_{\gamma})g_{\beta\alpha}\nonumber.
\end{eqnarray}
Now by using the relations $J_a\circ v=h\circ J_a$, $v\circ
J_a=J_a\circ h$ and \ref{torsion of D13}, we have
\begin{eqnarray}
T^{D^{a}}(vX,vY)&=&\nabla_{vX}vY-\nabla_{vY}vX-[vX,vY]\label{vDTorsion}\\
hT^{D^{a}}(hX,hY)&=&J_a(\nabla_{hY}J_ahX-\nabla_{hX}J_ahY+J_ah[hX,hY])\label{hDTorsion}.
\end{eqnarray}

Suppose that $X,Y\in\Gamma(T{\mathcal{U}}')$ are two arbitrary
vector fields on ${\mathcal{U}}'$ which have the following
representations in local coordinates:
\begin{eqnarray}
  X&=&z^\alpha\delta_\alpha+z^{n+\alpha}\delta_{n+\alpha}+w^\alpha\partial_\alpha+w^{n+\alpha}\partial_{n+\alpha}\\
  Y&=&\tilde{z}^\alpha\delta_\alpha+\tilde{z}^{n+\alpha}\delta_{n+\alpha}+\tilde{w}^\alpha\partial_\alpha+\tilde{w}^{n+\alpha}\partial_{n+\alpha}.
\end{eqnarray}
After performing some computations we have:
\begin{eqnarray}
  T^{D^a}(vX,vY)&=&S(vX,vY)\label{TDa}\\
  hT^{D^a}(hX,hY)&=&J_aT(J_ahY,J_ahX)\label{hTDa}.
\end{eqnarray}
The last equations show that $\nabla^a$ satisfies the second
condition of the theorem.\\
In this step we show that $\nabla^a$ is unique. Let
$\tilde{\nabla}^a$ be another linear connection on
$V{\mathcal{U}}'$ satisfying conditions $(1)$ and $(2)$.\\
By the equation $(\tilde{\nabla}^a_{vX}g)(vY,vZ)=0$ we have
\begin{eqnarray}
  vX(g(vY,vZ))=g(\tilde{\nabla}^a_{vX}vY,vZ)+g(vY,\tilde{\nabla}^a_{vX}vZ),
\end{eqnarray}
and so,
\begin{eqnarray}
  && vX(g(vY,vZ))+vY(g(vZ,vX))- vZ(g(vX,vY))=\nonumber \\
  &&\hspace*{1cm}g(2\tilde{\nabla}^a_{vY}vX+T^{\tilde{D}^a}(vX,vY)+[vX,vY],vZ)\label{unicue1}\\
  &&\hspace*{1cm}+g(T^{\tilde{D}^a}(vX,vZ)+[vX,vZ],vY)+g(T^{\tilde{D}^a}(vZ,vY)+[vZ,vY],vX),\nonumber
\end{eqnarray}
on the other hand we have,
$(\tilde{\nabla}^a_{hX})(vJ_aY,vJ_aZ)=0$, which shows that,
\begin{eqnarray}
  & &\hspace*{0.5cm}hX(g(vJ_aY,vJ_aZ))+hY(g(vJ_aZ,vJ_aX))-hZ(g(vJ_aX,vJ_aY))\label{unicue2}\\
  & & \hspace*{0.6cm}=g(J_ahT^{\tilde{D}^a}(hX,hY)+2\tilde{\nabla}^a_{hY}J_ahX+J_ah[hX,hY],J_ahZ)\nonumber\\
  & & \hspace*{1cm}+g(J_ahT^{\tilde{D}^a}(hX,hZ)+J_ah[hX,hZ],J_ahY)\nonumber\\
  & & \hspace*{1cm}+g(-J_ahT^{\tilde{D}^a}(hZ,hY)-J_ah[hZ,hY],J_ahX)\nonumber,
\end{eqnarray}
where $\tilde{D}^a$ is the linear connection induced by
$\tilde{\nabla}^a$ and theorem \ref{D}.\\
The relations \ref{unicue1} and \ref{unicue2} show that
$\tilde{\nabla}^a$ satisfies \ref{eq1} and \ref{eq2},
respectively. Therefore $\nabla^a=\tilde{\nabla}^a$.
\end{proof}

It is a natural question  whether we can use $J_2$ to introduce
a new connection similar to $J_1$ and $J_3$?\\
The answer is negative because $J_2$ dose not involve with the
non-linear connection and we can not introduce a connection $D^2$
in a similar way to $D^1$ and $D^3$ but we can have the following
theorem.
\begin{theorem}
Suppose that $\nabla$ is a Finsler connection on ${\mathcal{U}}'$.
Then the differential operator $D$ which is defined by
\begin{eqnarray}
  D_XY=\frac{1}{2}\{\nabla_XvY-J_1\nabla_XJ_1hY-J_2\nabla_XJ_2vY-J_3\nabla_XJ_3hY\},
\end{eqnarray}
where $X,Y\in\Gamma(T{\mathcal{U}}')$, is a linear connection on
${\mathcal{U}}'$. Also $J_1$, $J_2$ and $J_3$ are parallel with
respect to $D$.
\end{theorem}

\begin{proof}
For any $X, Y, Z \in\Gamma(T{\mathcal{U}}')$ and $f\in
C^\infty({\mathcal{U}}')$ we have
\begin{eqnarray}
  D_{fX+Y}Z&=&\frac{1}{2}\{f\nabla_XvZ+\nabla_YvZ-J_1(f\nabla_XJ_1hZ+\nabla_YJ_1hZ)\\
  &&-J_2(f\nabla_XJ_2vZ+\nabla_YJ_2vZ)-J_3(f\nabla_XJ_3hZ+\nabla_YJ_3hZ)\}\nonumber\\
  &=&fD_XZ+D_YZ\nonumber
\end{eqnarray}
\begin{eqnarray}
  D_X(fY+Z) &=& \frac{1}{2}\{Xf(vY)+f\nabla_XvY+\nabla_XvZ-J_1(Xf(J_1hY)\nonumber\\
  &&+f\nabla_XJ_1hY+\nabla_XJ_1hZ)-J_2(Xf(J_2vY)+f\nabla_XJ_2vY+\nabla_XJ_2vZ)\nonumber\\
  &&-J_3(Xf(J_3hY)+f\nabla_XJ_3hY+\nabla_XJ_3hZ)\}\\
  &=&\frac{1}{2}(2Xf(vY+hY))+fD_XY+D_XZ\nonumber\\
  &=&Xf(Y)+fD_XY+D_XZ.\nonumber
\end{eqnarray}
Therefore $D$ is a linear connection on ${\mathcal{U}}'$.\\
Now we show that the almost hypercomplex structure is parallel
with respect to $D$.\\
Let $Y$ be a vector field on ${\mathcal{U}}'$ which have the
representation
$Y=x^\alpha\delta_\alpha+y^\alpha\delta_{n+\alpha}+z^\alpha\partial_\alpha+w^\alpha\partial_{n+\alpha}$
in local coordinates. Then by a simple computation we have
\begin{eqnarray}
  (D_XJ_1)(Y) &=& D_X(J_1(Y))-J_1D_XY\nonumber \\
              &=&\frac{1}{2}\{\nabla_XvJ_1Y-J_1\nabla_XJ_1hJ_1Y-J_2\nabla_XJ_2vJ_1Y-J_3\nabla_XJ_3hJ_1Y\\
              &&-J_1\{\nabla_XvY-J_1\nabla_XJ_1hY-J_2\nabla_XJ_2vY-J_3\nabla_XJ_3hY\}\}\nonumber\\
              &=&0\nonumber
\end{eqnarray}
In a similar way we can show $(D_XJ_2)(Y)=(D_XJ_3)(Y)=0$. Hence
$J_1$, $J_2$ and $J_3$ are parallel with respect to $D$.
\end{proof}


\bibliographystyle{amsplain}

\end{document}